\documentclass[letterpaper,11pt]{amsart}

\usepackage{graphicx,color,epsfig}
\usepackage{amsfonts,amsmath,amssymb,amsthm}
\usepackage{comment}
\usepackage{hyperref}
\usepackage{xfrac}
\usepackage{fix-cm}
\usepackage{enumerate}
\usepackage{mathtools}
\usepackage{paralist}
\usepackage[normalem]{ulem} 

\newtheorem{thm}{Theorem}[section]

\newtheorem{lem}[thm]{Lemma}

\newtheorem{asm}{Assumption}

\theoremstyle{remark}
\newtheorem{rem}[thm]{Remark}

\theoremstyle{definition}
\newtheorem{defn}{Definition}

\newcommand{\ra}{\rightarrow}

\newcommand{\N}{\mathbb N}     
\newcommand{\R}{\mathbb R}     
\newcommand{\Z}{\mathbb Z}     

\renewcommand{\a}{\alpha}

\renewcommand{\d}{\delta}

\newcommand{\e}{\varepsilon}

\renewcommand{\l}{\lambda}

\newcommand{\s}{\sigma}

\renewcommand{\k}{\kappa}

\newcommand{\bigo}{\mathcal{O}}

\newcommand{\fl}[1]{\lfloor #1 \rfloor}  
\newcommand{\cl}[1]{\lceil #1 \rceil}    
\newcommand{\ind}[1]{ \mathbf{1}_{ \{ #1 \} } } 

\newcommand{\be}{\begin{equation}}
\newcommand{\ee}{\end{equation}}

\newcommand{\w}{\omega}              
\renewcommand{\P}{\mathbb{P}}        
\newcommand{\E}{\mathbb{E}}          
\newcommand{\vp}{\mathrm{v}_0}       

\DeclareMathOperator{\Var}{Var}     
\def\Pv{\mathbf{P}}  \def\Ev{\mathbf{E}}

\newcommand{\br}{\bar{r}}
\newcommand{\Ctr}{C_4}
\newcommand{\Cx}{C_0}
\newcommand{\Cza}{C_3}
\newcommand{\Czb}{C_z}
\newcommand{\Ch}{C_5}

\begin{document}

\title[Extreme slowdowns]{Extreme slowdowns for one-dimensional excited random walks}
\author{Jonathon Peterson}
\address{Jonathon Peterson \\  Purdue University \\ Department of Mathematics \\ 150 N University Street \\ West Lafayette, IN 47907 \\ USA}
\email{peterson@math.purdue.edu}
\urladdr{http://www.math.purdue.edu/~peterson}
\thanks{J. Peterson was partially supported by NSA grant H98230-13-1-0266.}

\subjclass[2010]{Primary: 60K35; Secondary: 60F10, 60J15, 60K37}
\keywords{excited random walk, large deviations}

\date{\today}


\begin{abstract}
 We study the asymptotics of the probabilities of extreme slowdown events for transient one-dimensional excited random walks. 
That is, if $\{X_n\}_{n\geq 0}$ is a transient one-dimensional excited random walk and $T_n = \min\{ k: \, X_k = n\}$, 
we study the asymptotics of probabilities of the form $P(X_n \leq n^\gamma)$ and $P(T_{n^\gamma} \geq n )$ with $\gamma < 1$. 
We show that there is an interesting change in the rate of decay of these extreme slowdown probabilities when $\gamma < 1/2$. 
\end{abstract}

\maketitle

\section{Introduction and main results}

Excited random walks are a model of self-interacting random walks where the the transition probabilities are a function of the local time of the random walk at the current location. 
The model of excited random walks was first introduced by Benjamini and Wilson in \cite{bwERW}, but has since been generalized by Zerner \cite{zMERW} and more recently by Kosygina and Zerner \cite{kzPNERW}. 
For the case of one-dimensional excited random walks, the model can be described as follows. 
A \emph{cookie environment} is an element $\w = \{ \w(x,j) \}_{x\in \Z, \, j\geq 1} \in [0,1]^{\Z\times \N} =: \Omega$.  
For a fixed cookie environment $\w$ we can define a self-interacting random walk on $\Z$ so that on the $j$-th visit of the random walk to the site $x$, the random walk steps to the right with probability $\w(x,j)$ and to the left with probability $1-\w(x,j)$. That is, $\{X_n\}_{n\geq 0}$ is a stochastic process with law $P_{\w}$ such that
\begin{align*}
  P_{\w}( X_{n+1}  = X_n + 1 \, | \, \mathcal{F}_n ) 
& = 1 -  P_{\w}( X_{n+1}  = X_n - 1 \, | \, \mathcal{F}_n ) \\
&= \w(X_n, \#\{k\leq n: X_k = X_n\} ), 
\end{align*}
where $\mathcal{F}_n = \s(X_0,X_1,\ldots,X_n)$. 
One can start the excited random walk at any $x\in \Z$, but in this paper we will always start the excited random walks at $X_0 = 0$.

We will allow the cookie environments to be random, chosen from a distribution $\P$ on the space $\Omega$ of cookie environments (equipped with the standard product topology). 
The distribution $P_\w$ of the random walk in a fixed cookie environment is the \emph{quenched} law of the random walk. Since the environment $\w$ is random, $P_\w$ is a conditional probability distribution, and the \emph{averaged} law $P$ of the excited random walk is defined by averaging the quenched law over all environments. That is, 
$P(\cdot) = \E\left[ P_\w(\cdot) \right]$, where $\E$ denotes expectation with respect to the distribution $\P$ on environments.


The terminology ``cookie environment'' is traced back to Zerner's paper \cite{zMERW} where he envisioned a stack of ``cookies'' at each site. Upon each visit to a site, the random walker eats a cookie (removing it from the stack) and the cookie induces a specific drift on the random walker\footnote{For this reason excited random walks are also sometimes called ``cookie random walks.''}.
Most of the results for one-dimensional excited random walks are under the assumption of a bounded number of cookies per site and i.i.d.\ stacks of cookies. 
More specifically, we will assume the following. 
\begin{asm}\label{asmM}
There exists an $M<\infty$ such that $\P(\w \in \Omega_M) = 1$, where 
\[
 \Omega_M = \{ \w \in \Omega: \, \w(x,j) = 1/2 \text{ for all } x\in \Z, \, j > M \}.
\]
\end{asm}

\begin{asm}\label{asmiid}
 The distribution $\P$ on cookie environments $\w$ is such that $\{\w(x,\cdot)\}_{x \in \Z}$ is i.i.d.\ under the measure $\P$. 
\end{asm}
\noindent Assumption \ref{asmM} is said to be the assumption of $M$ cookies per site because one imagines that after the $M$ cookies at a site have been removed, upon further returns to that site there are no cookies to ``excite'' the walk and so the walk moves as a simple symmetric random walk. Note that $\Omega_M$ is obviously isomorphic to the space $[0,1]^{\Z \times M}$. 

In addition to the above assumptions on the cookie environments, we will also need the following non-degeneracy assumption on the cookie environments. 
\begin{asm}\label{asmnd}
 The distribution $\P$ on cookie environments is such that 
\[
 \E\left[ \prod_{j=1}^M \w(0,j) \right] > 0, \quad \text{and} \quad \E\left[ \prod_{j=1}^M (1-\w(0,j)) \right] > 0. 
\]
\end{asm}

\subsection{Previous results}
Under Assumptions \ref{asmM}--\ref{asmnd} a great deal is known about the behavior of excited random walks. 
Remarkably, much of what is known can be characterized by a single parameter
\[
 \d = \d(\P) = \E\left[ \sum_{j = 1}^\infty (2 \w(0,j) - 1) \right].
\]
Since $2\w(0,j)-1$ is the expected displacement of the random walk in the step following the $j$-th visit to the origin, the parameter $\d$ is the ``average drift per site of the cookie environment.''
A brief summary of some of the ways in which the parameter $\d$ characterizes the behavior of the excited random walk is as follows. 

\textbf{Recurrence/transience.} The excited random walk is recurrent if $\d \in [-1,1]$, transient to the right if $\d>1$, and transient to the left if $\d<-1$ \cite{zMERW,kzPNERW}. 

\textbf{Law of large numbers/ballisticity.} There exists a constant $\vp$ such that 
\[
 \lim_{n\ra\infty} \frac{X_n}{n} = \vp, \quad P-a.s.
\]
Moreover, $\vp = 0$ if and only if $\d \in [-2,2]$ \cite{zMERW,bsCRWspeed,kzPNERW}. 

\textbf{Limiting distributions.} 
For any $\a\in(0,2]$ and $b>0$ let $\mathcal{Z}_{\a,b}$ denote an $\a$-stable random variable with characteristic exponent
\[
 \log E\left[ e^{iu \mathcal{Z}_{\a,b} } \right] = 
\begin{cases}
 -b|u|^\a\left( 1 - i \tan(\frac{\pi\a}{2})\frac{u}{|u|} \right) & \a \in (0,1) \cup (1,2] \\
 -b|u| \left( 1 + \frac{2i}{\pi}\frac{u}{|u|}\log|u|  \right) & \a = 1, 
\end{cases}
\qquad \forall u \in \R. 
\]
(Note that $\mathcal{Z}_{2,b}$ is a Gaussian random variable with mean $0$ and variance $2b$.)
If $\d>1$ so that the excited random walk is transient to the right, then the following limiting distributions are known \cite{bsRGCRW,kmLLCRW,kzPNERW}.  
\begin{align}
 &\text{If } \d \in (1,2), \text{ then}  &   &\frac{X_n}{n^{\d/2}} \xRightarrow[n\ra\infty]{} \left(\mathcal{Z}_{\d/2,b}\right)^{-\d/2} & \text{for some } b>0. \label{ERWld1} \\
&\text{If } \d \in (2,4), \text{ then} & & \frac{X_n - n \vp}{n^{2/\d}} \xRightarrow[n\ra\infty]{} \mathcal{Z}_{\d/2,b} & \text{for some } b>0. \label{ERWld2} \\
&\text{If } \d > 4, \text{ then} &    &\frac{X_n - n \vp}{\sqrt{n}} \xRightarrow[n\ra\infty]{} \mathcal{Z}_{2,b} & \text{for some } b>0. \label{ERWld3}
\end{align}
Note that the parameter $\d$ characterizes both the scaling needed and the type of the distribution in the limit. 
The limiting distributions in the borderline cases $\d \in \{2,4\}$ \cite{kmLLCRW,kzERWsurvey} and for recurrent ERW \cite{dCLTRERW,dkSLRERW} are also known. 
In these cases as well the parameter $\d$ determines both the scaling needed and the type of the limiting distribution. 

\subsection{Main results - asymptotics of extreme slowdown probabilities}

The results mentioned above show that for excited random walks that are transient to the right ($\d>1$), the polynomial rate of growth of $X_n$ is $n^{1 \wedge (\d/2)}$ (in the case $\d=2$ there is a logarithmic correction for the true rate of growth; that is, $(X_n \log n)/n$ converges in probability to a constant \cite{bsRGCRW,kmLLCRW}). 
In this paper we will be interested in the probability that the excited random walk grows at a much slower rate than this.

Large deviations of one-dimensional excited random walks were studied in \cite{pERWLDP}. 
Large deviation principles were proved both for the speed $X_n/n$ and for the hitting times $T_n/n$ where 
$T_n = \inf\{k\geq 0: X_k = n\}$. 
While large deviation probabilities of $X_n/n$ and $T_n/n$ generally decay exponentially in $n$, when the limiting speed $\vp>0$ (i.e., when $\d>2$) the probability of moving at a slower than typical positive speed decays polynomially in $n$. 
In particular, when $\d>2$ it was shown that 
\begin{equation}\label{eq:ldsd}
 \lim_{n\ra\infty} \frac{ \log P(X_n < nv) }{\log n} = \lim_{n\ra\infty} \frac{\log P( T_{n} > n/v ) }{\log n} = 1 - \frac{\d}{2}, \quad \forall v \in (0,\vp). 
\end{equation}

In this paper, we will be interested in the rate of decay of more extreme slowdown events; that is, where the random walk moves at a polynomial rate of growth that is slower than the typical rate of growth of $n^{1\wedge \d/2}$. 
To state the rates of decay we will use the following notation. For sequences $f(n)$ and $g(n)$ the notation $f(n) \sim g(n)$ as $n\ra\infty$ will mean that $f(n)/g(n) \ra 1$, while $f(n) \asymp g(n)$ will mean that the ratio $f(n)/g(n)$ is uniformly bounded away from $0$ and infinity. 
Our main result is the following. 
\begin{thm}\label{th:ERWxsd}
 Let Assumptions \ref{asmM}--\ref{asmnd} be satisfied, and let $\d > 1$. 
\begin{enumerate}
 \item \label{th:superpart} There exists a constant $A>0$ 
such that for any $\frac{1}{2} < \gamma < 1 \wedge \frac{\d}{2}$, 
\[
 P(T_{n^\gamma} > n ) \sim P( X_n < n^\gamma ) \sim A n^{\gamma - \d/2}, \quad \text{as } n\ra\infty.
\]
 \item \label{th:subpart} If $\gamma \in (0,1/2]$ then 
\[
 P(T_{n^\gamma} > n ) \asymp n^{2\gamma - \frac{1+\d}{2}}, \quad \text{and} \quad 
 P(X_n < n^\gamma ) \asymp n^{(1-\d)/2}. 
\]
\end{enumerate}
\end{thm}

\begin{rem}
 Here, and throughout the paper we will use the convention that $T_x = T_{\fl{x}}$ for $x \notin\Z$. We also will use this convention for other random variables indexed by integers when the index value given is not an integer. 
\end{rem}

\begin{rem}
 The formula for the exponent of the polynomial rate of decay for extreme slowdowns changes for both the position and the hitting times when $\gamma < 1/2$. Note, however, that in both cases the formula for the exponent is continuous since $\gamma - \d/2 = 2\gamma - (1+\d)/2 = (1-\d)/2$ when $\gamma = 1/2$. 
\end{rem}

Many of the results for transient, one-dimensional excited random walks are very similar to the corresponding results for random walks in random environments. 
While the parameter $\d$ characterizes much of the behavior of transient excited random walks, for random walks in random environments there is a parameter $\k>0$ of the distribution on the space of environments that plays a similar role. For many results in excited random walks, a similar statement is true for transient excited random walks with $\d/2$ replaced by $\k$. 
Examples of this are the limiting distributions in \eqref{ERWld1}--\eqref{ERWld3} and the large deviation slowdown asymptotics in \eqref{eq:ldsd} (see \cite{kksStable} and \cite{dpzAslowdown} for the corresponding results for RWRE). 

The asymptotics of extreme slowdown events such as $\{X_n < n^\gamma\}$ and $\{T_{n^\gamma} > n \}$ for transient one-dimensional RWRE were studied in \cite{fgpSDSU} where it was shown that (under the averaged measure) the probability of these events were $n^{\gamma-\k + o(1)}$ for any $\gamma \in (0,1 \wedge \k)$. 
The corresponding asymptotics for excited random walks in Theorem \ref{th:ERWxsd} above are similar when $\gamma \geq 1/2$ (again with $\d/2$ in place of $\k$), but are very different when $\gamma < 1/2$. 
This difference highlights an important difference in the way that extreme slowdowns occur in the excited random walks and RWRE. 
For RWRE, an extreme slowdown is typically caused by the random walk spending a long time (approximately $n$ steps) in a very short interval ($\bigo(\log n)$ in length) of the environment from which it is difficult for the random walk to escape. 
On the other hand, slowdowns in excited random walks are caused not by ``traps'' in the environment but by the self-interacting nature of the walk. 
Once the random walk visits all sites in an interval at least $M$ times, then the walk behaves as a simple symmetric random walk until it exits that interval.  
It will be seen from the proof below that extreme slowdowns occur when the random walk spends approximately $n$ steps in an interval of width of the order $ \sqrt{n}$. 
For this reason there is a change in the critical exponent of the slowdown asymptotics in Theorem \ref{th:ERWxsd} when $\gamma < 1/2$.

The outline of the paper is as follows. Section \ref{sec:BP} introduces the associated forward and backward branching process, details the relationship between these branching processes and the excited random walk, and reviews previous results for these branching processes. 
Then, in Section \ref{sec:super} we use the backward branching process to give the proof of the part \ref{th:superpart} of Theorem \ref{th:ERWxsd}.  Here we follow the approach from \cite{pERWLDP} that was used in proving the large deviation slowdown asymptotics in \eqref{eq:ldsd} but we make use of a stronger result on the tail asymptotics of sums of heavy tailed random variables to obtain not only the polynomial rate of decay but also the correct leading coefficient (see Remark \ref{rm:ldsdasym}). 
For the asymptotics of diffusive and subdiffusive slowdowns ($\gamma \in (0,1/2]$) this no longer works, and a different approach is needed. 
In Sections \ref{sec:subT} and \ref{sec:subX} we prove part \ref{th:subpart} of Theorem \ref{th:ERWxsd}. Our approach here uses both the backward branching process and the forward branching processes and requires some new probabilistic estimates for these branching processes. 

We close the introduction by detailing some conventions in notation that will be used throughout the remainder of the paper. 
First of all, since will consider a number of different stochastic processes throughout the paper it will be convenient to have a common notation for the stopping times of these processes. 
\begin{defn}
 If $\{\xi_t\}_{t \in I}$ is a stochastic process indexed by $I$ (the index set $I$ will always be either $\Z_+$ or $\R_+$), then
\[
 \s_x^\xi = \inf\{ t\in I: \, \xi_t \leq x \} \quad \text{and} \quad \tau_x^\xi = \inf\{ t \in I: \, \xi_t \geq x\}, \quad x \in \R,
\]
denote the first time for the process $\xi_t$ to be below or above the level $x$, respectively. (Note that the stopping times $\s_x^\xi$ or $\tau_x^\xi$ are equal to $+\infty$ if 
$\xi_t > x$ for all $t$ or $\xi_t < x$ for all $t$, respectively.)
\end{defn}
\noindent
Secondly, throughout the paper we will use $c,c',C,C',\ldots$ to denote arbitrary positive constants where the particular value of the constant may change from line to line. On the other hand, the numbered constants $C_0,C_1,C_2,\ldots$ will be used to denote particular constants that remain fixed throughout the paper. 

\section{Branching processes and diffusion process limits}\label{sec:BP}

One of the main tools that has been used in studying one-dimensional excited random walks is the \emph{associated branching processes with migration}. In this section, we introduce these branching processes, review how they relate to the excited random walk, and review some 
previous results for these branching processes. 

\subsection{The ``forward'' branching processes}

There are two branching processes with migration that are related to excited random walks. We will refer to these as the ``forward'' and ``backward'' braching processes, respectively. 
We will first consider the forward branching process. 

We begin by expanding the averaged measure $P(\cdot) = \E[P_\w(\cdot)]$ to include a family  
of Bernoulli random variables $\{B_{x,j}\}_{x\in \Z,j\geq 1}$ with the property that given the cookie environment $\w$ the $B_{x,j}$ are independent with $B_{x,j} \sim \text{Ber}(\w(x,j))$. 
We will use these Bernoulli random variables to construct both the path of the excited random walk and the associated forward branching processes. 
The random walk can be constructed by letting $B_{x,j}$ determine the jump of the random walk upon the $j$-th visit to a site $x$: if the random walk $X_n=x$ and $\#\{k\leq n: X_k = x \} = j$ then $X_{n+1} = X_n + 2(B_{x,j}-1)$.
The right forward branching process $\{W_i\}_{i\geq 0}$ and the left forward branching process $\{Z_i\}_{i\geq 0}$ are defined using the family of Bernoulli random variables as follows. 
Let $W_0 = w$ and $Z_0 = z$ for some fixed $w,z\geq 0$ and then for $i\geq 1$ let 
\begin{equation}\label{Wdef}
 W_i = \inf \left\{ m\geq 0: \sum_{j=1}^m (1-B_{i,j}) = W_{i-1} \right\} - W_{i-1}, 
\end{equation}
and
\begin{equation}\label{Zdef}
 Z_i = \inf \left\{m \geq 0: \sum_{j=1}^m B_{-i,j} = Z_{i-1} \right\} - Z_{i-1}. 
\end{equation}

The right and left forward branching processes constructed above are related to the excursions of the random walk away from the origin (to the right and left, respectively). 
To explain this connection we need to introduce some notation. First, let $\rho_n$ denote the time of the $n$-th return to the origin. That is, 
\[
 \rho_0 = 0, \quad \text{and} \quad \rho_n = \inf\{ k > \rho_{n-1} : \, X_k = 0 \}, \quad n\geq 1. 
\]
(Note that if the random walk is transient then eventually $\rho_n = \infty$ for all $n$ large enough). 
Also, for $x \in \Z$ and $n\geq 1$ let 
\[
 U_x^n = \# \{k \leq n: \, X_{k-1} = x, \, X_k = x+1 \} \quad \text{and}\quad D_x^n = \# \{k \leq n: \, X_{k-1} = x, \, X_k = x-1 \}. 
\]
That is, $U_x^n$ and $D_x^n$ are the number of steps to the right and left, respectively, of the random walk from the site $x$ in the first $n$ steps of the walk. 
Finally, for $n\geq 1$ let 
\begin{equation}\label{Rndef}
 R_n = \sum_{j=1}^n B_{0,j}. 
\end{equation}
Given the above notation, the following Lemma relates the forward branching processes to the excursions of the random walk from the origin. 

\begin{lem}\label{lem:fbpcorr}
If the right and left forward branching processes have the (random) initial conditions $W_0= R_n$ and $Z_0 = n-R_n$, then on the event $\{\rho_n < \infty\}$ we have that $W_i = U_i^{\rho_n}$ and $Z_i = D_{-i}^{\rho_n}$ for all $i\geq 0$. Moreover, $\{W_i\}_{i\geq 1}$ and $\{ Z_i \}_{i\geq 1}$ are conditionally independent given $W_0$ and $Z_0$. 
\end{lem}

\begin{proof}
It is easy to check that on the event $\{\rho_n < \infty\}$ the definition of $R_n$ implies that $U_0^{\rho_n} = R_n$. We will prove that $U_i^{\rho_n} = W_i$ for all $i\geq 0$ by induction. 
If $U_{i-1}^{\rho_n} = W_{i-1}$, then since $\rho_n < \infty$ all $W_{i-1}$ steps from $i-1$ to $i$ will later be followed by a corresponding down step from $i$ to $i-1$. Thus, the last step from site $i$ before time $\rho_n$ will be the $W_{i-1}$-th step from $i$ to $i-1$. On the other hand, it is easily seen from the construction of the process in \eqref{Wdef} that $W_i$ is equal to the number of steps from $i$ to $i+1$ before the $W_{i-1}$-th step from $i$ to $i-1$. This completes the proof that $U_i^{\rho_n} = W_i$ for all $i\geq 0$. The proof that $D_{-i}^{\rho_n} = Z_i$ for all $i\geq 0$ is similar. The conditional independence of $\{W_i\}_{i\geq 1}$ and $\{Z_i\}_{i\geq 1}$ follows easily from Assumption \ref{asmiid} and the fact that the Bernoulli random variables $\{B_{x,j}\}_{x\in\Z,j\geq 1}$ are conditionally independent given $\w$. 
\end{proof}

\begin{rem}\label{rem:fbpcorr}
The above identification of the forward branching processes 
with the first $n$ excursions of the random walk from the origin 
has the following useful consequences. (As in Lemma \ref{lem:fbpcorr} the following equalities hold on the event $\{\rho_n < \infty\}$ with the initial conditions $W_0=R_n$ and $Z_0 = n-R_n$ for the forward branching processes.)
\begin{enumerate}[1.]
 \item \textbf{Local times:} Let $L_n(x) = \# \{k < n: \, X_k = x\}$ be the local time process of the random walk. Then, 
\[
 L_{\rho_n}(x) = U_x^{\rho_n} + D_x^{\rho_n} 
= \begin{cases}
   W_0 + Z_0 & \text{ if } x = 0\\
   W_x + W_{x-1} & \text{ if } x > 0\\
   Z_{-x} + Z_{-x+1} & \text{ if } x < 0. 
  \end{cases}
\]
The second equality follows from the fact that $D_x^{\rho_n} = U_{x-1}^{\rho_n}$ for $x > 0$ and $U_x^{\rho_n} = D_{x+1}^{\rho_n}$ for $x<0$ on the event $\{\rho_n<\infty \}$. 
 \item \textbf{Range:} The range of the random walk during the first $n$ excursions from the origin is related to the ``lifetime'' of the forward branching processes. That is, 
\[
 \max_{k\leq \rho_n} X_k = \s_0^{W} \quad \text{and}\quad \min_{k\leq \rho_n} X_k = - \s_0^Z. 
\]
\end{enumerate}
\end{rem}

It follows from the fact that the cookie environment is i.i.d.\ (Assumption \ref{asmiid}) that the forward branching processes $\{W_i\}_{i\geq 0}$ and $\{Z_i\}_{i\geq 0}$ as constructed above are Markov chains.
While they were constructed by expanding the averaged measure $P$ for the excited random walk, for simplicity of notation 
we will let $P_W^w(\cdot)$ and $P_Z^z(\cdot)$ denote the marginal laws of the right and left forward branching processes starting with $W_0 = w$ and $Z_0 = z$, respectively. 
It is important to note that the forward branching processes are absorbing at $0$; that is, $W_i = 0$ for all $i \geq \s^W_0$ and $Z_i = 0$ for all $i\geq \s^Z_0$. 

We conclude this introduction to the forward branching processes with an explanation of why these are called ``branching processes.'' 
If the cookie environment $\w$ were such that $\w(x,j) \equiv p \in (0,1)$ (corresponding to a simple random walk), then the processes $\{W_i\}_{i\geq 0}$ and $\{Z_i\}_{i\geq 0}$ as defined above would be standard branching processes with offspring distributions Geo($1-p$) and Geo($p$), respectively. When the cookie environments instead satisfy Assumption \ref{asmM} it can be shown that the forward branching processes can be described by a critical branching process with migration where the offspring distribution is Geo($1/2$) and the migration law is determined by the law $\P$ of the cookie environment (c.f.\ a similar representation of the backward branching process in \cite[Section 2]{bsCRWspeed}).

\subsection{The ``backward'' branching process}

The backward branching process can also be defined in terms of the auxilliary Bernoulli random variables $\{B_{x,j}\}$ that were used in the construction of the forward branching processes. 
In particular, for a fixed $v \in \Z_+$ let $V_0 = 0$ and define 
\begin{equation}\label{Vdef}
 V_{i+1} = \inf\left\{ m\geq 0: \, \sum_{j=1}^m B_{i,j} = V_i+1 \right\} - V_i-1, \qquad \forall i \geq 0. 
\end{equation}
As with the forward branching processes, Assumption \ref{asmiid} implies that the backward branching process $\{V_i\}_{i\geq 0}$ is a Markov chain. 
We will use $P_V^v$ to denote the marginal law of the backward branching process started with initial condition $V_0 = v$. In contrast to the forward branching processes, the backward branching process is not absorbing at $0$. In fact, as we will see below the initial condition $V_0 = 0$ will be very important and so we will use the notation $P_V$ to denote $P_V^0$.

While the forward branching processes are related to the excursions of the random walk from the origin, the backward branching process is related to the hitting times $T_n$ of the random walk. 
In particular, from the construction of the excited random walk in terms of the Bernoulli random variables $\{B_{x,j}\}$ we have 
that  $D_n^{T_n} = 0$ and
\begin{equation}\label{DTnni}
 D_{n-i}^{T_n} =
\begin{cases}
 \inf \left\{m \geq 0: \sum_{j=1}^m B_{n-i,j} = D_{n-i+1}^{T_n} + 1 \right\} - (D_{n-i+1}^{T_n}+1) & i=1,2,\ldots n \\
 \inf \left\{m \geq 0: \sum_{j=1}^m B_{n-i,j} = D_{n-i+1}^{T_n} \right\} - D_{n-i+1}^{T_n} & i > n. 
\end{cases}
\end{equation}

\begin{rem}
 In \eqref{DTnni} we are implicitly using that $T_n < \infty$. However, since the results of this paper are for $\d>1$ and the excited random walk is transient to the right when $\d>1$, we have that $P(T_n < \infty) = 1$. 
\end{rem}

The above representation for $D_{n-i}^{T_n}$ when $i\leq n$ is very similar to the construction of the backward branching process in \eqref{Vdef} with the only difference being the sequence of Bernoulli random variables that are used. However, Assumption \ref{asmiid} implies that the $\{0,1\}^{\N}$-valued random variables $\mathbf{B}_x = (B_{x,j})_{j\geq 1}$ are i.i.d., and so it follows that 
the process $(D_n^{T_n},D_{n-1}^{T_n}, \dotsc,D_0^{T_n})$ has the same distribution as the first $n$ generations of the backward branching process $(V_0,V_1,\ldots,V_n)$ started with $V_0 = 0$.

\begin{rem}\label{rem:immigrant}
In interpreting \eqref{DTnni} as a branching process, the ``children'' correspond to steps to the left of the random walk. Most of the steps the the left from a site $x$ can be thought of as ``descendents'' of a previous jump to the left from the site $x+1$. However, since the initial visit to a site $x\geq 0$ doesn't come from a previous jump to the left from $x+1$ there are some ``children'' in these generations that don't correspond to a ``parent'' from the previous generation. In branching process terminology, the first $n$ generations of the process have an extra immigrant before reproduction. 
\end{rem}

Comparing \eqref{Zdef} with \eqref{DTnni} when $i>n$, one obtains that $\{D_{n-i}^{T_n}\}_{i\geq n}$ has the same distribution as a left forward branching process $\{Z_i\}_{i\geq 0}$ with (random) initial condition $Z_0 = D_0^{T_n}$. 
For this reason it will be useful to expand the measure $P_V^v$ for the backward branching process in the following way: 
for any $n\geq 1$, $\{Z^{(n)}_i\}_{i\geq 0}$ will be a Markov chain with the following properties
\begin{enumerate}
 \item $P_V^v( Z_0^{(n)} = V_n ) = 1$. 
 \item $\{Z_i^{(n)}\}_{i\geq 0}$ is a left forward branching process. That is, 
\[
 P_V^v( \{Z_i^{(n)}\}_{i\geq 0} \in \cdot \, | \, V_n = z) = P_Z^z( \{Z_i\}_{i\geq 0} \in \cdot) \quad \text{for any } z \geq 0. 
\]
 \item Since the only change in the branching structure of $\{D_{n-i}^{T_n}\}_{i\geq 0}$ when $i>n$ is the lack of the extra immigrant before reproduction (see Remark \ref{rem:immigrant} above), we can couple the processes $\{Z_i^{(n)}\}_{i\geq 0}$ with the backward branching process $\{V_i\}_{i\geq 0}$ in such a way that 
\begin{equation}\label{VZcouple}
 P_V^v\left( Z_i^{(n)} \leq V_{n+i} \text{ for all } i\geq 0 \right) = 1. 
\end{equation}
\end{enumerate}
In summary, since $T_n = n + 2 \sum_{i\geq 0} D_{n-i}^{T_n}$ on the event $\{T_n < \infty\}$, we can conclude that 
\begin{equation}\label{TnVrel}
 P(T_n = k) = P_V\left( n + 2\sum_{i=0}^n V_i + 2\sum_{i\geq 1} Z_i^{(n)} = k\right), \quad \forall k < \infty. 
\end{equation}

\subsection{Diffusion process limits}
In this section we will review some of the limiting distributions for the forward and backward branching processes constructed above. 
We begin by introducing the diffusion processes that arise as the scaling limits. For any $\a \in \R$ and $y > 0$, let $Y_\a(t)$ be a solution to the stochastic differential equation 
\begin{equation}\label{BSQdef}
 dY_\a(t) = \a \, dt + \sqrt{2|Y_\a(t)| } \, dB(t), 
\quad Y_\a(0) = y, 
\end{equation}
where $\{B(t)\}_{t\geq 0}$ is a standard one-dimensional Brownian motion. 
We will use $P_{Y_\a}^y(\cdot)$ to denote the law of the diffusion process $\{Y_\a(t)\}_{t \leq \s_0^{Y_\a}}$ started at $Y_\a(0) = y$.
\begin{rem}
It should be noted that $2 Y_\a(t)$ is a squared Bessel process of dimension $2\a$; see \cite[Chapter XI]{ryCMBM} and \cite{gySGBP} for more information on squared Bessel processes. 
We will note here only a few basic properties. First of all, it is known that $Y_\a(t)$ has a unique strong solution for any $\a \in \R$ and $y>0$. 
Secondly, the hitting time $\s_0^{Y_\a}$ of 0 is almost surely finite if $\a < 1$ and infinite if $\a\geq 1$. 
Finally, if $\a \geq 0$ then the process $Y_\a(t) \geq 0$ for all $t$, while if $\a < 0$ then $Y_\a(t) $ is instead non-positive for $t \geq \s_0^{Y_\a}$. However, this will not be important in the current paper since we will only be interested in the diffusion $Y_\a(t)$ for $t \leq \s_0^{Y_\a}$.
\end{rem}

In \cite{kmLLCRW} it was observed that the backward branching process is well behaved when the population size is away from 0. In particular, 
for any $k\geq M$
\[
 E[V_{i+1}-V_i \, | \, V_i = k ] = 1-\d, \quad \text{and} \quad 
\Var(V_{i+1}-V_i \, | \, V_i=k ) = s + 2(V_i - M  + 1), 
\]
where $s>0$ is the variance of some specified non-degenerate random variable. 
These two facts suggest that when the population size is large that the (rescaled) backward branching process is similar to the diffusion $Y_{1-\d}(t)$. In fact, Kosygina and Mountford proved that this is true if the backward branching process is stopped sufficiently far away from 0. 
\begin{thm}[Lemma 3.1 in \cite{kmLLCRW}]\label{th:bbpdpl}
 Let Assumptions \ref{asmM}--\ref{asmnd} hold with $\d>0$, and let $\{\xi_n\}_{n\geq 1} = \{\xi_n(i)\}_{n\geq 1, \,i\geq 0}$ be a sequence of backward branching processes with initial conditions $\xi_n(0) = y_n$ such that $y_n/n \ra y > 0$ as $n\ra\infty$. 
Then, for any $\e>0$
\begin{equation}\label{eq:dpl}
 \left\{ \frac{\xi_n(\fl{nt} \wedge \s_{\e n}^{\xi_n} )}{n} \right\}_{t\geq 0 } 
\xRightarrow[n\ra\infty]{J_1}
\{Y_{1-\d}(t \wedge \s_\e^{Y_{1-\d}})\}_{t \geq 0}, \quad \text{with } Y_{1-\d}(0) = y,
\end{equation}
where $\overset{J_1}{\Longrightarrow}$ denotes convergence in distribution on the space of cadlag processes with the Skorohod-$J_1$ topology. 
\end{thm}

The techniques of Mountford and Kosygina were recently improved and applied to the forward branching processes in \cite{kzEERW}. For the left and right forward branching processes, it is known that 
\[
 E[Z_{i+1}-Z_i \, | \, Z_i = k ] = -\d, \quad \text{and} \quad E[W_{i+1}-W_i \, | \, W_i = k ] = \d, 
\]
for all $k>M$ (see \cite[Lemma 10]{kzPNERW}), while the variance of the increments are similar to those of the backward branching process. Thus, the corresponding diffusion process limits for the left and right forward branching processes are $Y_{\text{-}\d}(t)$ and $Y_\d(t)$, respectively. 
In \cite{kzEERW}, Kosygina and Zerner also studied the scaling limits for the forward branching process conditioned to die out. 
When $\d>1$, both the forward branching process $W_i$ and the diffusion limit $Y_\d(t)$ are transient (since excited random walk is transient to the right, excursions to the right of the origin can potentially never return to the origin). It is easy to see that for any $\e>0$ the process $Y_\d(\cdot \wedge \tau_\e^{Y_\d} )$  conditioned on the event $\{\tau_\e^{Y_\d} < \infty\}$ has the same distribution as the process $Y_{2-\d}(\cdot \wedge \tau_\e^{Y_{2-\d}})$. In \cite{kzEERW} it was shown that the forward branching process conditioned on the event $\{\s_0^W < \infty\}$ converges in distribution (after rescaling) to the diffusion process $Y_{2-\d}(\cdot \wedge \s_0^{Y_{2-\d}} )$. 
The diffusion limits for the forward branching processes proved in \cite{kzEERW} are stronger than the diffusion limit for the backward branching process in Theorem \ref{th:bbpdpl}. In particular, the diffusion limits in \cite{kzEERW} hold all the way down to $\s_0^{Y_\a}$ and limiting distribution are obtained for the lifetime and the total progeny of the forward branching processes.

\begin{thm}[Theorem 14 in \cite{kzEERW}]\label{th:fbpdpl}
 Let Assumptions \ref{asmM}--\ref{asmnd} hold with $\d>1$. If $\{\xi_n\}_{n\geq 1} = \{\xi_n(i)\}_{n\geq 1, \,i\geq 0}$ is a sequence of either 
\begin{inparaenum}[1)]
 \item left forward branching processes, or
 \item right forward branching processes conditioned to die out,
\end{inparaenum}
and the initial conditions $\xi_n(0) = y_n$ are such that $y_n/n \ra y>0$ as $n\ra\infty$, then 
\[
 \left\{ \frac{\xi_n(\fl{nt})}{n} \right\}_{t \geq 0} 
\xRightarrow[n\ra\infty]{J_1}
\{Y_{\a}(t \wedge \s_0^{Y_\a} )\}_{t\geq 0},  \quad \text{with }Y_{\a}(0) = y,
\]
where the value of $\a$ is 
\begin{inparaenum}[1)]
 \item $\a=-\d$, or
 \item $\a=2-\d$
\end{inparaenum}
based on which type of branching processes the $\xi_n$ are (left forward, or conditioned right forward, respectively). 
Moreover, the following following limiting distributions hold for the lifetime and total progeny of the branching processes. 
\[
 \frac{\s_0^{\xi_n}}{n} \xRightarrow[n\ra\infty]{} \s_0^{Y_\a} \quad \text{and} \quad 
\frac{1}{n^2} \sum_{i\geq 0} \xi_n(i)  \xRightarrow[n\ra\infty]{} \int_0^{\s_0^{Y_\a}} Y_\a(t) \, dt, 
\]
where $\Rightarrow$ signifies convergence in distribution and either \begin{inparaenum}[1)]
 \item $\a=-\d$, or
 \item $\a=2-\d$,
\end{inparaenum}
depending on the type of branching process that $\xi_n$ is. 
\end{thm}

\begin{rem}
 It is suspected that the stronger conclusions in Theorem \ref{th:fbpdpl} also hold for the backward branching process, although Theorem \ref{th:bbpdpl} is sufficient for our purposes in the current paper. In fact the analysis of the forward and backward branching processes is very similar and many of the proofs carry over almost word for word (see \cite[Remark 9]{kzEERW}). We also note that the results in \cite{kzEERW} cover a wider range of $\d$ than as stated in Theorem \ref{th:fbpdpl} above. For simplicity we only consider the case $\d>1$ since this is what is needed for the current paper. 
\end{rem}

\subsection{Additional branching process results}

Before proceeding to the proof of the main results of the current paper, we will recall some previous results for the forward and backward branching processes that will be useful in the remainder of the paper. 
We begin with a simple monotonicity property of the branching processes. 
\begin{lem}\label{bpmono}
 For any fixed integers $1\leq z\leq z'$, one can couple two left forward branching processes  $\{Z_i\}_{i\geq 0}$ and $\{ Z_i' \}_{i\geq 0}$ so that $Z_0 = z$, $Z_0'=z'$ and $Z_i \leq Z_i'$ for all $i\geq 0$. 
A similar statement is true for the right forward branching process and the backward branching process. 
\end{lem}

\begin{proof}
 The monotonicity follows easily from the construction of the branching processes. In particular, recalling the family of Bernoulli random variables $\{B_{x,j}\}_{x\in\Z, j\geq 1}$, if we regard $B_{x,j} = 1$ as a ``success'' and $B_{x,j} = 0$ as a ``failure'' then $Z_i$ is the number of failures in the sequence $\{B_{-i,j}\}_{j\geq 1}$ before the $Z_{i-1}$-th success. If we construct the branching processes $Z_i$ and $Z_i'$ using the same family of Bernoulli random variables, then the desired coupling follows immediately. 
The monotonicity for the other branching processes is similar. 
\end{proof}

Next, we will give some asymptotics of hitting probabilities for the left forward branching process.
The hitting probabilities for the limiting diffusion are easily obtained by noting that $(Y_{\text{-}\d}(t))^{1+\d}$ is a martingale. 
While this fact is not used in the proofs of the following two Lemmas, it does give insight as to how they are proved. 

\begin{lem}\label{fbphitprob}
 There exists $\ell_0<\Z_+$ and constants $0<c<C$ such that 
\[
 c 2^{-(u-m)(\d+1)} \leq P_Z^{2^m}( \s_{2^\ell}^Z > \tau_{2^u}^Z ) \leq C 2^{-(u-m)(\d+1)}, 
\quad \text{for integers } \ell_0\leq \ell < m < u. 
\]
\end{lem}
\begin{proof}
Upper and lower bounds for the corresponding hitting probabilities of the backward branching process were given in \cite[Lemma 5.3]{kmLLCRW}. 
However, the same proof follows through essentially word for word for the left forward branching process $Z_i$ by replacing $\d$ with $\d+1$ everywhere. In particular, there exists a $\l>0$ and an $\ell_*>0$ such that
\[
 \frac{h_\ell^-(m) - 1}{h_\ell^-(u) - 1} \leq P_Z^{2^m}( \s_{2^\ell}^Z > \tau_{2^u}^Z) ) \leq  \frac{h_\ell^+(m) - 1}{h_\ell^+(u) - 1}, 
\quad \text{for all } \ell_*\leq \ell < m < u,
\]
where the functions $h_\ell^{\pm}$ are given by $h_\ell^{\pm}(k) = \prod_{i=\ell+1}^k (2^{\d+1}\mp 2^{-\l i})$ for $k > \ell$. 
It is easy to see that if $\ell$ is sufficiently large then 
\[
 2^{(\d+1)(k-\ell-1)} \leq h_\ell^+(k) \leq 2^{(\d+1)(k-\ell)} \leq h_\ell^-(k) \leq 2^{(\d+1)(k-\ell+1)}, \quad \forall k > \ell.
\]
From this the conclusion of the lemma follows easily. 
\end{proof}

\begin{lem}\label{fbphpasymp}
For any fixed $z \geq 1$, there exists a constant $C>0$ (depending on $z$) such that 
\[
 P_Z^z( \tau_n^Z < \s_0^Z) \leq C n^{-(\d+1)}, \quad \forall n \geq 1. 
\]
\end{lem}
\begin{proof}
The corresponding statement for the backward branching process was proved in \cite[Corollary 5.5]{kmLLCRW}, however, the same proof works for the left forward branching process when $\d>-1$. 
\end{proof}

The final results for the branching processes that we will recall are specific to the backward branching process. 
Unlike the forward branching processes, the backward branching process $V_i$ is an irreducible Markov chain since it is not absorbing at 0 (this is due to the extra ``immigrant'' before reproduction in every generation). Therefore, we can introduce a regeneration structure for the backward branching process. Let $r_0^V = \s_0^V$ be the first time the process $V_i$ reaches $0$ (if $V_0 = 0$ then $r_0^V = 0$). Successive returns to 0 are then given by 
\[
 r_k^V = \inf\{ i > r_{k-1}^V : \, V_i = 0 \}, \quad \forall k\geq 1. 
\]
A priori, it could be that $r_k^V = \infty$ for some $k$. However, it is known that when $\d>1$ then the backward branching process is recurrent and so all regeneration times $r_k^V$ are finite with probability 1. 
Another important random variable associated with the backward branching process is the total progeny in the branching process between regeneration times. That is, 
\[
 S_k^V = \sum_{i=r_{k-1}^V}^{r_k^V-1} V_i, \quad \text{for } k\geq 1. 
\]

\begin{lem}[Theorems 2.1 and 2.2 in \cite{kmLLCRW}]\label{lem:reght}
 If $\d>0$, then there exist constants $C_1,C_2 > 0$ such that 
\[
 P_V(r_1^V > n) \sim C_1 n^{-\d}, \quad \text{as } n\ra\infty, 
\]
and 
\[
 P_V(S_1^V > n) \sim C_2 n^{-\d/2}, \quad \text{as } n\ra\infty. 
\]
\end{lem}

\begin{rem}
 While regeneration times do not exist for the forward branching processes, there are similar tail asymptotics for the lifetime and total progeny for the forward branching processes in \cite[Theorem 21]{kzEERW}. However, those asymptotics will not be needed in the present paper. 
\end{rem}

Since the excited random walks in the current paper are assumed to have $\d>1$, Lemma \ref{lem:reght} implies that $E_V[r_1^V] < \infty$. Since this parameter will arise frequently in the proofs below, for convenience of notation we will denote the mean regeneration time of the backward branching process by $\br = E_V[r_1^V]$. 

\section{Superdiffusive slowdowns}\label{sec:super}

Having given the necessary background on the forward and backward branching processes associated to the excited random walk, we are now ready to give the proofs of the main results of the paper. 
In this section we will prove part \ref{th:superpart} of Theorem \ref{th:ERWxsd}. That is, we will compute the asymptotics of $P(X_n < n^\gamma)$ and $P(T_{n^\gamma} > n)$ when $\gamma \in (1/2, \d/2 \wedge 1)$. 

We first study the asymptotics of $P(T_{n^\gamma} > n)$ using the backward branching process as a key tool. 
It follows from \eqref{TnVrel} that $T_{n^\gamma}$ stochastically dominates $\fl{n^\gamma} + 2 \sum_{i=0}^{\fl{n^\gamma}} V_i$ when $V_0 = 0$. If the $k$-th regeneration time $r_k^V$ occurs by the $\fl{n^\gamma}$-th generation of the branching process then $\sum_{i=0}^{\fl{n^\gamma}} V_i \geq \sum_{i=0}^{r_k^V} V_i =  \sum_{j=1}^k S_j^V$. Thus, we can conclude for any $s>0$ that 
\begin{equation}\label{TnVlb}
 P( T_{n^\gamma} > n ) \geq 
P_V\left( \sum_{j=1}^{\fl{sn^{\gamma}}} S_j^V > \frac{n - \fl{n^\gamma}}{2} \right) - P_V\left( r_{\fl{sn^\gamma}}^V > \fl{n^\gamma} \right). 
\end{equation}
To compute the asymptotics of the two probabilities on the right side above we will need the following Lemma. 
\begin{lem}\label{lem:htsums}
 Let $\xi_1,\xi_2,\xi_3,\ldots$ be i.i.d.\ non-negative random variables with $P(\xi_1 > t) \sim \Cx t^{-\a}$ as $t\ra\infty$ for some $\Cx>0$ and $\a>0$. 
\begin{enumerate}
 \item If $\a>1$ so that $\bar\xi:=E[\xi_1] < \infty$, then for any $x > \bar\xi$, \label{lem:htsumspt1}
\[
  P\left( \sum_{i=1}^{n} \xi_i > x n \right) 
\sim \Cx (x-\bar\xi)^{-\a} n^{1-\a}, \quad \text{as } n\ra\infty. 
\]
 \item If $0<\gamma < \a \wedge 1$, then \label{lem:htsumspt2}
\[
 P\left( \sum_{i=1}^{n^{\gamma}} \xi_i > x n \right) 
\sim \Cx x^{-\a} n^{\gamma-\a}, \quad \text{as } n\ra\infty. 
\]
\end{enumerate}
\end{lem}
\begin{rem}
 Sums of i.i.d.\ random variables are very well studied and much stronger results on the asymptotics of sums of heavy tailed random variables are contained in \cite{bbAARW}. 
In fact, Lemma \ref{lem:htsums} follows from Theorems 2.6.1, 3.4.1, and 4.4.1 in \cite{bbAARW} in the cases $\a\in (0,1)$, $\a \in (1,2)$, and $\a >2$, respectively. 
In the boundary cases $\a=1$ or $\a=2$, the asymptotics in Lemma \ref{lem:htsums} can be deduced from Corollaries 2.2.4, 2.5.2, and 3.1.7 and Theorem 3.3.1 in \cite{bbAARW}. 
\end{rem}

Lemma \ref{lem:htsums} is applicable to the right side of \eqref{TnVlb} since the  $\{S_j^V\}_{j\geq 1}$ are i.i.d.\ and the $k$-th regeneration time $r_k^V$ is the sum of $k$ independent copies of $r_1^V$. 
Therefore, from Lemmas \ref{lem:reght} and \ref{lem:htsums} we can conclude that if $\gamma < \d/2 \wedge 1$ then
\begin{equation}\label{SVsum}
 P_V\left( \sum_{j=1}^{\fl{sn^{\gamma}}} S_j^V > \frac{n - \fl{n^\gamma}}{2} \right) \sim C_2 s 2^{\d/2} n^{\gamma-\d/2} , \quad \text{as } n\ra\infty. 
\end{equation}
Similarly, recalling the definition of $\br = E_V[r_1^V]$, we have that if $1/s > \br$ then 
\begin{equation}\label{rktail}
 P_V\left( r_{\fl{sn^\gamma}}^V > \fl{n^\gamma} \right) \sim C_1 s (1-s\br)^{-\d} n^{\gamma(1-\d)}, \quad \text{as } n\ra\infty.
\end{equation}
Note that $n^{\gamma(1-\d)} = o(n^{\gamma-\d/2})$ when $\gamma > 1/2$. Therefore, by applying \eqref{SVsum} and \eqref{rktail} to \eqref{TnVlb} and optimizing over $s$ we can conclude that 
\begin{equation}\label{superTnlb}
 \liminf_{n\ra\infty} n^{-\gamma+\d/2} P( T_{n^\gamma} > n ) \geq \frac{C_2 2^{\d/2}}{\br}, \qquad \forall \gamma \in (1/2,\d/2 \wedge 1). 
\end{equation}

A corresponding upper bound for $P(T_{n^\gamma} > n)$ is obtained similarly. 
First, note that the coupling in \eqref{VZcouple} implies that 
$\sum_{i\geq 1} Z_i^{(n^\gamma)} \leq \sum_{i=\fl{n^\gamma}+1}^{r_k^V} V_i$ if the $k$-th regeneration time $r_k^V > n^\gamma$. 
Therefore, from \eqref{TnVrel} we can conclude that 
\begin{equation}\label{TnVub}
 P(T_{n^\gamma} > n ) \leq P_V\left( \fl{n^\gamma} + 2 \sum_{j=1}^{\fl{sn^\gamma}} S_j^V  > n \right) + P_V\left( r_{\fl{sn^\gamma}}^V \leq n^\gamma \right).
\end{equation}
The asymptotics of the first probability on the right are given by \eqref{SVsum}. For the second probability on the right, since $r_{\fl{sn^\gamma}}^V$ is the sum of $\fl{sn^\gamma}$ i.i.d.\ non-negative random variables with finite mean, we can conclude from Cramer's Theorem that  
\begin{equation}\label{rklefttail}
 \lim_{n\ra\infty} \frac{1}{n^\gamma} \log  P_V\left( r_{\fl{sn^\gamma}}^V \leq n^\gamma \right) < 0, \quad \text{for any } s > 1/\br. 
\end{equation}
Therefore, applying \eqref{SVsum} and \eqref{rklefttail} to \eqref{TnVub} and optimizing over $s$ we can conclude that 
\begin{equation}\label{superTnub}
 \limsup_{n\ra\infty} n^{-\gamma+\d/2} P( T_{n^\gamma} > n ) \leq \frac{C_2 2^{\d/2}}{\br}, \qquad \forall \gamma < \d/2 \wedge 1. 
\end{equation}
Combining \eqref{superTnlb} and \eqref{superTnub} proves part Theorem \ref{th:ERWxsd}\ref{th:superpart} for the hitting times
with $A = \frac{C_2 2^{\d/2}}{\br}$. 

We now move to the proof of Theorem \ref{th:ERWxsd}\ref{th:superpart} for $X_n$. A lower bound for the asymptotics follows easily from the above asymptotics for the hitting time slowdowns. Indeed, since $\{X_n < n^\gamma\} \supset \{ T_{\fl{n^\gamma}} > n \}$, we can conclude from \eqref{superTnub} that 
\[
 \liminf_{n\ra\infty} n^{-\gamma +\d/2} P(X_n < n^\gamma ) \geq \frac{C_2 2^{\d/2}}{\br} \qquad \forall \gamma \in (1/2,\d/2 \wedge 1). 
\]
The corresponding upper bound is slightly more complicated since the event $\{X_n < n^\gamma \}$ can occur even if the random walk has travelled very far past $n^\gamma$ in the first $n$ steps. 
However, since the random walk is transient to the right, it is unlikely that the random walk backtracks too far during the first $n$ steps. To make this precise, let $\e>0$ and note that 
\[
 P(X_n < n^\gamma) \leq P\left( T_{(1+\e)n^\gamma}  > n \right) + P\left( \inf_{k > T_{(1+\e)n^\gamma}} X_k \leq n^\gamma \right). 
\]
It was shown in \cite[Lemma 6.1]{pERWLDP} that there exists a $C>0$ such that $P\left( \inf_{k\geq T_{n+m}} X_k \leq n \right) \leq C m^{1-\d}$ for all $n,m\geq 1$. From this it follows that the second probability on the right above is $\bigo( n^{\gamma(1-\d)})$ for any fixed $\e>0$.
Since $n^{\gamma(1-\d)} = o(n^{\gamma - \d/2})$ when $\gamma > 1/2$, we can conclude (by repeating the argument leading to \eqref{superTnub}) that 
\[
 \limsup_{n\ra\infty} n^{-\gamma+\d/2} P(X_n < n^\gamma) \leq \limsup_{n\ra\infty} n^{-\gamma+\d/2}P\left( T_{(1+\e)n^\gamma}  > n \right) = \frac{C_2 2^{\d/2}(1+\e)}{\br}, 
\]
for any $\gamma \in (1/2,\d/2\wedge 1)$ and $\e>0$. Taking $\e\ra 0$ completes the proof of Theorem \ref{th:ERWxsd}\ref{th:superpart} for the position $X_n$ of the excited random walk. 

\begin{rem}\label{rm:ldsdasym}
 The above proof of Theorem \ref{th:ERWxsd}\ref{th:superpart} uses essentially the same arguments as the proof in \cite{pERWLDP} of the large deviation slowdown asymptotics \eqref{eq:ldsd}. 
However, Lemma \ref{lem:htsums} above gives better control on the large deviations of sums of heavy tailed random variables than what was used in \cite{pERWLDP}. 
By applying Lemma \ref{lem:htsums} to the proof in \cite{pERWLDP} one can easily obtain 
following improved asymptotics for the large deviation slowdowns: if $\d>2$, then for any $v \in (0,\vp)$, 
\[
 \lim_{n\ra\infty} n^{\d/2-1} P(X_n < nv) =  \lim_{n\ra\infty} n^{\d/2-1} P(T_{n/v} > n) = \frac{C_2 (2\vp)^{\d/2}}{\br} v(\vp - v)^{-\d/2}, \quad \forall v \in (0,\vp). 
\]
\end{rem}

\section{Diffusive and sub-diffusive slowdowns for hitting times}\label{sec:subT}
In this section we will prove part \ref{th:subpart} of Theorem \ref{th:ERWxsd} for the hitting times. That is, we will show that 
\[
 P(T_{n^\gamma} > n ) \asymp n^{2\gamma - \frac{1+\d}{2}}, \quad \forall \gamma \in (0,1/2]. 
\]

\subsection{Lower bound}
The connection between hitting times and branching processes in \eqref{TnVrel} implies that $T_{n^\gamma}$ stochastically dominates $2\sum_{i\geq 0} Z^{(n^\gamma)}_i$, where for $n$ fixed $\{Z^{(n^\gamma)}_i\}_{i\geq 0}$ is a left forward branching process with initial condition given by $Z_0^{(n^\gamma)} = V_{n^\gamma}$.  Therefore, we can conclude that 
\begin{equation}\label{Tnsublb}
 P(T_{n^\gamma} > n ) \geq P_V\left( \sum_{i=0}^\infty Z_i^{(n^\gamma)} > n \right) \geq P_V( V_{n^\gamma} \geq \fl{ n^\gamma} ) P_Z^{\fl{n^\gamma}}\left( \sum_{i=0}^\infty Z_i > n \right),
\end{equation}
where in the second inequality we used the monotonicity of the branching processes with respect to the initial condition from Lemma \ref{bpmono}. 

To evaluate the probabilities on the right in \eqref{Tnsublb}, we will use the following two lemmas.
\begin{lem}\label{PVnn}
 If $\d>1$, 
then there exists a constant $c>0$ such that $P_V(V_n \geq n ) \geq c n^{1-\d}$ for all $n$ large enough. 
\end{lem}

\begin{lem}\label{Zsumtail}
For any $0<a < b < 1/2$, there are constants $c,C>0$ such that for all $n$ large enough, 
 \[
  c z^{1+\d} n^{-\frac{1+\d}{2}} \leq P_Z^z\left( \sum_{i=0}^\infty Z_i > n \right)
\leq C z^{1+\d} n^{-\frac{1+\d}{2}},
\quad \forall z \in (n^a,n^b). 
 \]
\end{lem}

\begin{rem}
 It follows from \cite[Theorem 21]{kzEERW} that there exists a constant $\Cza>0$ such that $P_Z^1( \sum_{i\geq 0} Z_i > n ) \sim \Cza n^{-\frac{1+\d}{2}}$. Moreover, the proof can probably be extended to show that for any fixed $z \in \N$ that there exists a constant $\Czb>0$ such that $P_Z^z( \sum_{i\geq 0} Z_i > n ) \sim \Czb n^{-\frac{1+\d}{2}}$. The bounds in Lemma \ref{Zsumtail} are similar, but instead allow for control of the probabilities as the initial value $Z_0 = z$ increases at a polynomial rate slower than $\sqrt{n}$. 
\end{rem}

Postponing the proofs of these lemmas for the moment, we note that they imply that if $\gamma < 1/2$ then 
there exists a constant $c>0$ such that for all $n$ large enough
\[
 P(T_{n^\gamma} > n) \geq c n^{\gamma (1-\d)} n^{\gamma(1+\d)-\frac{1+\d}{2}} = c n^{2\gamma - \frac{1+\d}{2}}. 
\]
When $\gamma = 1/2$, we can no longer apply Lemma \ref{Zsumtail} to the second probability on the right side of \eqref{Tnsublb}.
Instead, we can apply Theorem \ref{th:fbpdpl} to obtain that  
\begin{equation}\label{ZsumYint}
 \lim_{n\ra\infty} P_Z^{\fl{\sqrt{n}}} \left( \sum_{i = 0}^\infty  Z_i > n \right) = P_{Y_{\text{-}\d}}^1\left( \int_0^{\s_0^{Y_{\text{-}\d}}} Y_{\text{-}\d}(t) \, dt > 1 \right) > 0.
\end{equation}
Therefore, applying Lemma \ref{PVnn} to \eqref{Tnsublb} we obtain that there exists a constant $c>0$ such that  $P(T_{\sqrt{n}} > n) \geq c n^{\frac{1-\d}{2}}$ for all $n$ large enough. 
Since $2\gamma - \frac{1+\d}{2} = \frac{1-\d}{2}$ when $\gamma = 1/2$, this gives the required lower bound in the case $\gamma = 1/2$. 

To complete the proof of the lower bound for $P(T_{n^\gamma} > n )$ when $\gamma \in (0,1/2]$ it remains to give the proofs of Lemmas \ref{PVnn} and \ref{Zsumtail}.

\begin{proof}[Proof of Lemma \ref{PVnn}]

We begin by noting that
for any $\e>0$, 
\begin{align}
 P_V(V_n \geq n ) 
&\geq P_V\left( \tau_{\e n}^V \leq \frac{n}{2} , \, \inf_{i\in [\frac{n}{2}, n]} V_{\tau_{\e n}^V+i} \geq n \right) \nonumber \\
&\geq P_V( \tau_{\e n}^V \leq n/2 ) P_V^{\fl{\e n}} \left( \inf_{i \in [\frac{n}{2},n]} V_i \geq n \right),\label{Vnnlb} 
\end{align}
where in the second inequality we used the strong Markov property and the monotonicity of the branching process in the initial condition from Lemma \ref{bpmono}. 
For the second probability on the right in \eqref{Vnnlb}, the diffusion approximation in Theorem \ref{th:bbpdpl} implies that for any fixed $\e>0$
\begin{align}
 \liminf_{n\ra\infty} P_V^{\fl{\e n}} \left( \inf_{i \in [\frac{n}{2},n]} V_i \geq n \right) 
&\geq \lim_{n\ra\infty} P_V^{\fl{\e n}} \left( \inf_{i < \frac{n}{2}} V_i \geq \e n/2 , \, \inf_{i \in [\frac{n}{2},n]} V_i \geq  n \right) \nonumber \\
&= P_{Y_{1-\d}}^\e\left( \inf_{t < 1/2} Y_{1-\d}(t) \geq \e/2, \, \inf_{t \in [\frac{1}{2},1] } Y_{1-\d}(t) \geq  1 \right) > 0. \label{infVrange}
\end{align}
To control the first probability in \eqref{Vnnlb}, note that for any fixed $s>0$
\begin{equation}\label{tVlb}
 P_V( \tau_{\e n}^V \leq n/2) = P_V\left( \max_{i \leq n/2} V_i \geq \e n \right) \geq P_V\left( \max_{i \leq r_{\fl{sn}}^V} V_i \geq \e n \right) - P_V\left( r_{\fl{sn}}^V > n/2 \right). 
\end{equation}
For the first term on the right in \eqref{tVlb}, it was shown in 
\cite[Lemma 8.1]{kmLLCRW} that there exists a $\Ctr > 0$ such that 
\begin{equation}\label{PVtr}
  P_V( \tau_m^V < r_1^V ) \sim \Ctr m^{-\d}, \quad \text{as } m\ra\infty,
\end{equation}
and from this it is easy to see that
\[
 P_V\left( \max_{i \leq r_{\fl{sn}}^V} V_i \geq \e n \right) = 1 - \left( 1 - P_V( \tau_{\e n}^V < r_1^V ) \right)^{\fl{sn}} \sim \Ctr s \e^{-\d} n^{1-\d}, \quad \text{as } n\ra\infty. 
\]
For the second term on the right in \eqref{tVlb}, Lemmas \ref{lem:reght} and \ref{lem:htsums} imply that 
\[
 P_V\left( r_{\fl{sn}}^V > n/2 \right) \sim C_1 s\left(\tfrac{1}{2}-s \br \right)^{-\d} n^{1-\d}, \quad \text{ as } n\ra\infty \quad \text{if } 
s < \frac{1}{2\br}. 
\]
Therefore, by choosing $s < 1/(2\br)$ and then fixing $\e>0$ small enough so that $ \Ctr \e^{-\d} > C_1 (\frac{1}{2}-s \br)^{-\d}$, we can conclude that 
\begin{equation}\label{tauVearly}
 \liminf_{n\ra\infty} n^{-1+\d} P_V( \tau_{\e n}^V < n/2) > 0. 
\end{equation}
Applying \eqref{infVrange} and \eqref{tauVearly} to \eqref{Vnnlb}, we conclude that 
$P_V( V_n \geq n) \geq c n^{1-\d}$ for all $n$ large enough.
\end{proof}

\begin{proof}[Proof of Lemma \ref{Zsumtail}] 
As a first step in the proof of Lemma \ref{Zsumtail}, we will show that there are constants $c,C>0$ such that for all $n$ large enough,
\begin{equation}\label{Pupfromz}
 c z^{1+\d}n^{-\frac{1+\d}{2}} \leq P_Z^z\left( \tau_{\sqrt{n}}^Z < \s_0^Z \right) \leq C z^{1+\d}n^{-\frac{1+\d}{2}}, \quad \forall z \in (n^a,n^b). 
\end{equation}
For convenience of notation we will let $m = \cl{\log_2 z}$ and $u = \fl{ \log_2 \sqrt{n} }$ (note that if $n$ is large enough and $z \in (n^a,n^b)$ then $\ell_0+1<m<u$, where $\ell_0$ is from Lemma \ref{fbphitprob}).
With this notation, we have that Lemmas \ref{fbphitprob} and \ref{fbphpasymp} imply that there are constants $c,C>0$ such that 
\begin{align*}
c 2^{-(u-m+2)(1+\d)} & \leq P_Z^{2^{m-1}} \left( \tau_{2^{u+1}}^Z < \s_{2^{\ell_0}}^Z \right) \\
& \leq P_Z^z\left( \tau_{\sqrt{n}}^Z < \s_0^Z \right) \\
& \leq P_Z^{2^m}\left( \tau_{2^u}^Z < \s_{2^{\ell_0}}^Z \right) + P_Z^{2^{\ell_0}}\left( \tau_{\sqrt{n}}^Z < \s_0^Z \right)
\leq C 2^{-(u-m)(1+\d)} + C n^{-\frac{1+\d}{2} }.
\end{align*}
Since $z^{1+\d}n^{-\frac{1+\d}{2}} \leq 2^{-(u-m)(1+\d)} \leq 4^{1+\d} z^{1+\d}n^{-\frac{1+\d}{2}}$ this completes the proof of \eqref{Pupfromz}.

Now we turn to the asymptotics of $P_Z^z( \sum_{i=0}^\infty Z_i > n )$. For a lower bound, first note that the strong Markov property, Lemma \ref{bpmono} and \eqref{Pupfromz} imply that 
\begin{align*}
 P_Z^z\left( \sum_{i=0}^\infty Z_i > n \right) 
&\geq P_Z^z\left( \tau_{\sqrt{n}}^Z < \s_0^Z \right) P_Z^{\fl{\sqrt{n}}} \left( \sum_{i=0}^\infty Z_i > n \right) \\
&\geq c z^{1+\d}n^{-\frac{1+\d}{2}}P_Z^{\fl{\sqrt{n}}} \left( \sum_{i=0}^\infty Z_i > n \right), 
\end{align*}
for all $n$ large enough and $z \in (n^a,n^b)$. 
The required lower bound in the statement of the lemma is then obtained by noting that 
\eqref{ZsumYint} implies that the last probability on the right is bounded away from 0 for all $n$ large enough. 

For a corresponding upper bound for $P_Z^z( \sum_{i=0}^\infty Z_i > n )$, first note that \eqref{Pupfromz} implies that
\begin{align}
 P_Z^z\left(  \sum_{i=0}^\infty Z_i > n \right) 
& \leq P_Z^z( \tau_{\sqrt{n}}^Z < \s_0^Z ) + P_Z^z\left( \sum_{i=0}^{\s_0^Z} \ind{Z_i < \sqrt{n}} > \sqrt{n} \right) \nonumber \\
& \leq C z^{1+\d}n^{-\frac{1+\d}{2}} + P_Z^z\left( \sum_{i=0}^{\s_0^Z} \ind{Z_i < \sqrt{n}} > \sqrt{n} \right), \quad \forall z \in (n^a,n^b). \label{Zsumtailub}
\end{align}
To bound the probability on the right in \eqref{Zsumtailub}, again let $m = \cl{\log_2 z}$ and $u = \fl{ \log_2 \sqrt{n} }$. Then, it follows from 
\cite[Lemma 18]{kzEERW} that there exists a constant $c>0$ such that 
\begin{align}
& P_Z^z\left( \sum_{i=0}^{\s_0^Z} \ind{Z_i < \sqrt{n}} > \sqrt{n} \right) \nonumber \\
&\qquad \leq \sum_{j=1}^{u+1} P_Z^z\left( \sum_{i=0}^{\s_0^Z} \ind{ Z_i \in [ 2^{u-j+1},2^{u-j+2}) } > \frac{ 2^{j-2}}{j(j+1)} 2^{u-j+2} \right) \nonumber \\
&\qquad \leq \sum_{j=1}^{u+1} \exp\left\{ -c  \frac{ 2^{j-2}}{j(j+1)}  \right\} P_Z^z\left( \inf\{ i\geq 0: \, Z_i \in [2^{u-j+1},2^{u-j+2}) \} < \s_0^Z \right) \nonumber \\
&\qquad \leq \sum_{j=1}^{u-m} \exp\left\{ -c  \frac{ 2^{j-2}}{j(j+1)}  \right\}P_Z^{2^m}( \tau_{2^{u-j+1}}^Z < \s_0^Z ) + \sum_{j=u-m+1}^{u+1} \exp\left\{ -c  \frac{ 2^{j-2}}{j(j+1)}  \right\}, \label{vcbound}
\end{align}
where in the first inequality we used that $2^u \leq \sqrt{n}$ and $\sum_{j=1}^u \frac{1}{j(j+1)} < 1$. 
Now, another application of Lemmas \ref{fbphitprob} and \ref{fbphpasymp} similar to the proof of \eqref{Pupfromz} implies that there exists a constant $C>0$ such that
\[
 P_Z^{2^m}( \tau_{2^{u-j+1}}^Z < \s_0^Z ) \leq C 2^{-(u-j-m+1)(1+\d)} \leq C z^{1+\d} n^{-\frac{1+\d}{2}} 2^{j(1+\d)}. 
\]
Therefore, the first sum in \eqref{vcbound} is bounded above by 
\[
 C z^{1+\d} n^{-\frac{1+\d}{2}} \sum_{j=1}^{u-m+1} 2^{j(1+\d)} \exp\left\{ -c  \frac{ 2^{j-2}}{j(j+1)}  \right\} \leq C' z^{1+\d} n^{-\frac{1+\d}{2}}.
\]
For the second sum in \eqref{vcbound}, note that the terms inside the sum are decreasing in $j$ (if $j>2$) and so for $n$ sufficiently large we can bound the sum by
\begin{align*}
 (m+1) \exp\left\{ -c  \frac{ 2^{u-m}}{(u-m+2)(u-m+3)}  \right\}
& \leq (1+\cl{\log_2 z})\exp\left\{ -c \frac{ \sqrt{n}}{4z (\log_2(\frac{\sqrt{n}}{z})+3)^2}  \right\} \\
&\leq C (\log_2 n) \exp\left\{ -c' \frac{n^{1/2-b}}{ (\log_2 n)^2 } \right\}. 
\end{align*}
Since this last line is less than $C n^{a(1+\d)} n^{-\frac{1+\d}{2}} \leq C z^{1+\d} n^{-\frac{1+\d}{2}}$ for all $n$ large enough and $z > n^a$, this completes the proof of the upper bound for $P_Z^z( \sum_{i=0}^\infty Z_i > n)$. 
\end{proof}

\subsection{Upper bound}
In this subsection, we will show that $P(T_{n^\gamma} > n )= \bigo(n^{2\gamma-\frac{1+\d}{2}} )$ when $\gamma \in (0,1/2]$. 
Since $\d>1$, the upper bound in \eqref{superTnub} implies that $P(T_{n^\gamma} > n) \leq C n^{\gamma(1-\d)}$ for some $C>0$ and all $n$ large enough for any $\gamma \in (0,1/2]$. 
This gives the required uppper bound when $\gamma = 1/2$ since $\gamma(1-\d)= 2\gamma - \frac{1+\d}{2}$ when $\gamma = 1/2$. 
However, when $\gamma < 1/2$ then $\gamma(1-\d) > 2\gamma - \frac{1+\d}{2}$ and so we need a different argument to get a better upper bound. 

Since \eqref{TnVrel} implies that 
\[
 P(T_{n^\gamma} > n ) = P_V\left( \fl{n^{\gamma}} + 2 \sum_{i=0}^{n^\gamma-1} V_i + 2 \sum_{i=0}^\infty Z_i^{(n^\gamma)} > n \right),
\]
when $\gamma \in (0,1/2)$ it will be enough to show that 
\begin{equation}\label{sububc1}
 P_V\left( \sum_{i=0}^{n^\gamma-1} V_i > \frac{n}{5} \right) = o(n^{2\gamma-\frac{1+\d}{2}} ),
\end{equation}
and
\begin{equation}\label{sububc2}
 P_V\left( \sum_{i=0}^\infty Z_i^{(n^\gamma)} > \frac{n}{4} \right) = \bigo(n^{2\gamma-\frac{1+\d}{2}} ). 
\end{equation}
To show \eqref{sububc1}, note that \eqref{PVtr} implies that
\[
 P_V\left( \sum_{i=0}^{n^\gamma-1} V_i > \frac{n}{5} \right)
\leq P_V\left( \tau_{\frac{n^{1-\gamma}}{5}}^V  < n^\gamma \right) \leq n^\gamma P_V\left( \tau_{\frac{n^{1-\gamma}}{5}}^V < r_1^V \right) = \bigo( n^{\gamma-\d(1-\gamma)} ).  
\]
Since ${\gamma-\d(1-\gamma)} < 2\gamma - \frac{1+\d}{2}$ when $\d>1$ and $\gamma < 1/2$, this completes the proof of \eqref{sububc1}.  
To prove \eqref{sububc2} we first fix some $a \in (0,\frac{2\gamma}{1+\d})$ and  $b \in (\gamma,1/2)$. 
Then, by conditioning on the value of $V_{n^\gamma} = Z_0^{(n^\gamma)}$ and applying Lemma \ref{Zsumtail} we obtain that 
\begin{align}
 P_V\left( \sum_{i=0}^\infty Z_i^{(n^\gamma)} > \frac{n}{4} \right)
&= E_V\left[ P_Z^{V_{n^\gamma}}\left(  \sum_{i=0}^\infty Z_i > \frac{n}{4} \right) \right] \nonumber\\
&\leq P_Z^{n^{a}} \left(  \sum_{i=0}^\infty Z_i > \frac{n}{4} \right) 
 + P_V( V_{n^\gamma} > n^{b} ) \nonumber\\
&\qquad + E_V\left[ P_Z^{V_{n^\gamma}}\left(  \sum_{i=0}^\infty Z_i > \frac{n}{4} \right) \ind{V_{n^\gamma} \in (n^{a},n^{b}) } \right] \nonumber \\
&\leq o\left(n^{2\gamma-\frac{1+\d}{2}} \right) + P_V( V_{n^\gamma} > n^{b} ) + C n^{-\frac{1+\d}{2}} E_V\left[ (V_{n^\gamma})^{1+\d} \right]. \label{ZsumVng}
\end{align}
Thus, we need to obtain bounds on $P_V( V_{n^\gamma} > n^b )$ and $E_V[ (V_{n^\gamma})^{1+\d} ]$. To this end, we will use the following Lemma.
\begin{lem}\label{PVnyn}
 There exist constants $c,C>0$ such that 
\[
 P_V( V_n > yn ) \leq 
\begin{cases}
 C n^{1-\d} y^{-\d} & y \in (0,4] \\
 C n^{1-\d} e^{-c y}&  y >4.
\end{cases}
\]
\end{lem}
We postpone the proof of Lemma \ref{PVnyn} momentarily and note that it implies that there are constants $c,C,C'>0$ such that for $n$ large enough
\begin{equation}\label{Vngbounds}
 P_V( V_{n^\gamma} > n^{b} ) \leq C n^{1-\d} e^{-c n^{b-\gamma}}, 
\quad\text{and}\quad
 E_V\left[ (V_{n^\gamma})^{1+\d} \right] \leq C' n^{2\gamma}. 
\end{equation}
The first inequality in follows immediately from Lemma \ref{PVnyn} since $b>\gamma$. For the second inequality, Lemma \ref{PVnyn} implies that 
\begin{align*}
 E_V\left[ (V_{n^\gamma})^{1+\d} \right] 
&= (1+\d) n^{\gamma(1+\d)} \int_0^\infty y^\d P_V( V_{n^\gamma} > y n^\gamma ) \, dy \\
&\leq (1+\d) n^{\gamma(1+\d)} \left\{ \int_0^4 y^\d \left( C n^{\gamma(1-\d)}y^{-\d} \right) \, dy + \int_4^\infty y^\d \left( C n^{\gamma(1-\d)}e^{-cy} \right) \, dy \right\} \\
&\leq C' n^{2\gamma}. 
\end{align*}
Combining \eqref{ZsumVng} and \eqref{Vngbounds} we obtain \eqref{sububc2} which, as noted at the beginning of the section, together with \eqref{sububc1} implies that $P(T_{n^\gamma} > n) = \bigo( n^{2\gamma - \frac{1+\d}{2}})$ when $\gamma < 1/2$. It remains however to give the proof of Lemma \ref{PVnyn}. 

\begin{proof}[Proof of Lemma \ref{PVnyn}]
For an easy upper bound on $P_V(V_n > yn)$, first note that 
\[
 P_V(V_n > yn) \leq P_V\left( \max_{i \leq r_n^V} V_i > yn \right) \leq n P_V\left( \max_{i\leq r_1^V} V_i > yn \right) \leq n P_V( \tau_{yn}^V < r_1^V )
\leq C n^{1-\d}y^{-\d}, 
\]
where the last inequality follows from \eqref{PVtr}. 
This upper bound holds for any $y>0$, but we will need a better bound for $y$ large. 
To this end, first note that \eqref{PVtr} and \cite[Lemma 5.1]{kmLLCRW} imply that there exist constants $C,c>0$ such that for $n$ large enough and $y>4$, 
\begin{align}
 P_V( V_n > y n )
& \leq P_V\left( \tau_n^V \leq n, \, V_{\tau_n^V} > \frac{yn}{2} \right) + P_V\left( \tau_n^V \leq n, \, V_{\tau_n^V} \leq \frac{yn}{2}, \, V_n > yn \right) \nonumber \\
& \leq n P_V\left( \tau_n^V < r_1^V, \,  V_{\tau_n^V} > \frac{yn}{2} \right) + P_V( \tau_n^V \leq n ) P_V^{\frac{yn}{2}}\left( \tau_{yn}^V \leq n \right) \nonumber \\
&\leq n P_V( \tau_n^V < r_1^V ) \left\{  P_V\left( V_{\tau_n^V} > \frac{yn}{2} \, \bigl| \, \tau_n^V < r_1^V \right) + P_V^{\frac{yn}{2}}\left( \tau_{yn}^V \leq n \right) \right\} \nonumber \\
&\leq C n^{1-\d} \left\{  e^{-c y} + P_V^{\frac{yn}{2}}\left( \tau_{yn}^V \leq n \right) \right\}. \label{Ptauyn}
\end{align}
To obtain an upper bound for $P_V^{\frac{yn}{2}}\left( \tau_{yn}^V \leq n \right)$, we first recall that $E_V^v[V_1-v] = 1-\d$ for all $v \geq M$ (see \cite[Lemma 3.3]{bsCRWspeed} or \cite[Lemma 17]{kzPNERW}). On the other hand, for $v < M$ we have that $E_V^v[ V_1-v ] \geq -v$, and so we can conclude that there exists a constant $0<A \leq (\d-1)\wedge M$ such that $E_V^v[ V_1 - v ] \geq - A$ for all $v\geq 0$. 
From this we can conclude that $\{V_i + Ai\}_{i\geq 0}$ is a submartingale under the natural filtration $\mathcal{F}_i^V = \s(V_0,V_1,\ldots,V_i)$. 
Then, the maximal inequality for submartingales implies that 
\begin{align}
 P_V^{\frac{yn}{2}}\left( \tau_{yn}^V \leq n \right)
&= P_V^{\frac{yn}{2}}\left( \max_{i\leq n} V_i \geq yn \right) \nonumber \\
&\leq P_V^{\frac{yn}{2}}\left( \max_{i\leq n} (V_i + Ai) \geq yn \right) \nonumber \\
&\leq e^{-y/4} e^{A/4} E_V^{\frac{yn}{2}} \left[ e^{\frac{V_n}{4n} }\right]. \label{maxineq}
\end{align}
To bound the expectation in the last line we will use the following Lemma. 
\begin{lem}\label{Vngf}
 For any $n\geq 1$ and $v \geq 0$, 
\[
 E_V^v\left[ s^{V_n} \right] \leq s^M \left( \frac{1}{n-(n-1)s} \right)^{M+1} \left( \frac{n-(n-1)s}{n+1-ns} \right)^{v+1}, \quad \forall s \in [0,1+1/n). 
\]
\end{lem}
We postpone for the moment the proof of Lemma \ref{Vngf}, and note that since $e^{1/(4n)} < 1+1/n$ for all $n$ sufficiently large we have that 
\begin{equation}\label{mgfub}
  E_V^{\frac{yn}{2}} \left[ e^{\frac{V_n}{4n} }\right] \leq e^{\frac{M}{4n}} \left( \frac{1}{n-(n-1)e^{\frac{1}{4n}}} \right)^{M+1} \left( \frac{n-(n-1)e^{\frac{1}{4n}} }{n+1-n e^{\frac{1}{4n}} } \right)^{1+\frac{yn}{2}}
\leq 2 \left( \frac{4}{3} \right)^{M+1} e^{y/5}, 
\end{equation}
where the last inequality follows from the following limits that can easily be checked.  
\[
 \lim_{n\ra\infty} n-(n-1)e^{\frac{1}{4n}} = \frac{3}{4}, 
\quad\text{and}\quad
 \lim_{n\ra\infty}  \left( \frac{n-(n-1)e^{\frac{1}{4n}} }{n+1-n e^{\frac{1}{4n}} } \right)^{1+\frac{n}{2}} = e^{1/6}. 
\]
Combining \eqref{maxineq} and \eqref{mgfub}, we obtain that $P_V^{\frac{yn}{2}}\left( \tau_{yn}^V \leq n \right) \leq C e^{-y/20}$. 
Recalling \eqref{Ptauyn}, this completes the proof of Lemma \ref{PVnyn}, pending the proof of Lemma \ref{Vngf}. 
\end{proof}

\begin{proof}[Proof of Lemma \ref{Vngf}]
We begin by noting that
if $V_{n-1} = k$ then $V_n$ is the number of ``failures'' before the $(k + 1)$-th ``success'' in the sequence $\{B_{n-1,j} \}_{j\geq 1}$ (compare with the similar statement for the backward branching process in the proof of Lemma \ref{bpmono}). Obviously this is bounded above by $M$ plus the number of failures before the $(k+1)$-th success in the sequence $\{ B_{n-1,j} \}_{j\geq M+1}$. Since the $\{B_{n-1,j}\}_{j\geq M+1}$ are i.i.d.\ Bernoulli($1/2$) we can conclude that given $V_{n-1}$ the random variable $V_n$ is stochastically dominated by $M +\sum_{k=1}^{V_{n-1}+1} \zeta_k$  where the $\zeta_k$ are i.i.d.\ Geometric($1/2$) random variables. 
Let $\phi(s) = \Ev[s^{\zeta_1}] = \frac{1}{2-s}$ be the probability generating function for the geometric random variables (note that $\phi(s) < \infty$ if $s\in [0,2)$). Then, we have that 
\[
 E_V^v\left[ s^{V_n} \right] \leq s^M E_V^v\left[ \phi(s)^{V_{n-1}+1} \right] = s^M \phi(s) E_V^v\left[ \phi(s)^{V_{n-1}} \right]. 
\]
Iterating this $n$ times we obtain that  
\[
 E_V^v\left[ s^{V_n} \right] \leq s^M \left( \prod_{k=1}^{n-1} \phi^{(k)}(s) \right)^{M+1} \left( \phi^{(n)}(s) \right)^{1+v},
\]
where $\phi^{(k)}(s) = \phi(\phi^{(k-1)}(s))$ is the function $\phi$ composed with itself $k$ times. The proof of the Lemma is then completed by noting that $\phi^{(k)}(s) = \frac{k-(k-1)s}{k+1-ks}$, which can be easily checked by induction (note also that $\phi^{(k)}(s) < \infty$ if $s \in [0,1+1/k)$). 
\end{proof}

\section{Diffusive and sub-diffusive slowdown probabilities for \texorpdfstring{$X_n$}{position}}\label{sec:subX}

In this section we will prove the slowdown asymptotics in Theorem \ref{th:ERWxsd}\ref{th:subpart} for the position of the excited random walk. That is, we will show that
\begin{equation}\label{Xnsubasymp}
 P(X_n \leq n^\gamma) \asymp n^{\frac{1-\d}{2}}, \quad \text{for any } \gamma \in (0,1/2]. 
\end{equation}
 For the upper bound, note that 
\begin{equation}\label{Xnsubub}
 P(X_n \leq n^\gamma) 
\leq P(X_n \leq \sqrt{n}) 
\leq P(T_{(1+\e)\sqrt{n}} > n ) + P\left( \inf_{k\geq T_{(1+\e)\sqrt{n}}} X_k \leq \sqrt{n} \right). 
\end{equation}
Using the same argument for the upper bound for hitting times in Theorem \ref{th:ERWxsd}\ref{th:subpart} we can obtain that $P(T_{(1+\e)\sqrt{n}} > n ) \leq C n^{\frac{1-\d}{2}}$ for some $C>0$ and all $n$ large enough. 
For the second probability on the right in \eqref{Xnsubub}, 
it follows from \cite[Lemma 6.1]{pERWLDP} that the second probability on the right above is bounded by $C n^{\frac{1-\d}{2}}$ for some constant $C>0$ as well. This proves the upper bound for $P(X_n \leq n^{\gamma} )$ in \eqref{Xnsubasymp}

For the lower bound in \eqref{Xnsubasymp} 
it will be enough to show that $P(X_n \leq 0 ) \geq c n^{\frac{1-\d}{2}}$ for some $c>0$. 
To this end, recall that $\rho_k$ denotes the time of the $k$-th return to the origin of the excited random walk, and that $L_n(x) = \# \{k < n: \, X_k = x\}$ dentoes the local time of the excited random walk.
Then, letting $\Pv_{\sfrac{1}{2}}$ denote the law of a simple symmetric random walk $\{S_n\}_{n\geq 0}$ started at $S_0=0$, we have that for any $\e>0$
\begin{align*}
 P(X_n \leq 0) 
& \geq P\left( \rho_{\sqrt{n}} \leq n, \, \min_{|i|\leq \e \sqrt{n}} L_{\rho_{\sqrt{n}}}(i) \geq M \right) \left\{ \min_{k\leq n} \Pv_{\sfrac{1}{2}}\left( S_k \leq 0, \, \max_{j\leq k} |S_j| \leq  \e \sqrt{n} \right) \right\}, 
\end{align*}
since once the excited random walk has visited all sites in an interval at least $M$ times it then moves as a simple symmetric random walk until exiting that interval. 
It follows from Donsker's invariance principle that for any $\e>0$ the minimum inside the braces on the right is bounded away from 0 as $n\ra\infty$, and thus we only need to show that for some $\e>0$, 
\begin{equation}\label{Xreturns}
\liminf_{n\ra\infty} n^{\frac{\d-1}{2}} P\left( \rho_{\sqrt{n}} \leq n, \, \min_{|i|\leq \e \sqrt{n}} L_{\rho_{\sqrt{n}}}(i) \geq M \right) > 0. 
\end{equation}

We will prove \eqref{Xreturns} by using the left and right forward branching processes to study the excursions of the random walk away from the origin. 
Recall that when $\rho_k < \infty$, $R_k = U_0^{\rho_k}$ of the first $k$ excursions from the origin are to the right and $k-R_k = D_0^{\rho_k}$ are to the left.
Then, Lemma \ref{lem:fbpcorr} and Remark \ref{rem:fbpcorr} imply that
\begin{align}
& P\left( \rho_{\sqrt{n}} \leq n, \, \min_{|i|\leq \e \sqrt{n}} L_{\rho_{\sqrt{n}}}(i) \geq M \right) \nonumber \\
&\quad \geq \sum_{m=0}^{\sqrt{n}} P\left( R_{\fl{\sqrt{n}}} = m \right) P_W^m\left( 2 \sum_{i=0}^\infty W_i \leq \frac{n}{2}, \, \s_M^W > \e \sqrt{n} \right)  \nonumber \\
&\qquad\qquad\qquad \times 
P_Z^{\fl{\sqrt{n}}-m}\left( 2 \sum_{i=0}^\infty Z_i \leq \frac{n}{2}, \, \s_M^Z > \e \sqrt{n} \right) \nonumber \\
&\quad \geq P\left( \left| R_{\fl{\sqrt{n}}} - \frac{\sqrt{n}}{2} \right| \leq n^{3/8} \right)
\left\{ \min_{|w-\frac{\sqrt{n}}{2}|\leq n^{3/8} } P_W^w\left( \sum_{i=0}^\infty W_i \leq \frac{n}{4}, \, \s_M^W > \e\sqrt{n} \right) \right\} \label{unifbp} \\
&\qquad\qquad\qquad \times \left\{ \min_{|z-\frac{\sqrt{n}}{2}|\leq n^{3/8} } P_Z^z\left( \sum_{i=0}^\infty Z_i \leq  \frac{n}{4}, \, \s_M^Z > \sqrt{n} \right) \right\}. \nonumber 
\end{align}
For the first probability on the right, we claim that 
\begin{equation}\label{Rndev}
 \lim_{n\ra\infty} P( |R_{\fl{\sqrt{n}}} - \sqrt{n}/2 | \leq n^{3/8} ) = 1. 
\end{equation}
To see this, recall from the construction of $R_k$ in \eqref{Rndef} that for any $k\geq M$,
\begin{itemize}
 \item $R_M$ and $R_k-R_M$ are independent. 
 \item $R_k-R_M$ is a Binomial($k-M,1/2$) random variable. 
 \item $R_M$ is a bounded random variable. 
\end{itemize}
Therefore, we can conclude that $E[ R_{\fl{\sqrt{n}}} ] = \sqrt{n}/2 + \bigo(1)$ and $\Var(R_{\fl{\sqrt{n}}}) = \sqrt{n}/4 + \bigo(1)$, and from this \eqref{Rndev} follows easily. 

It remains to consider the asymptotics of the two minimums of probabilities on the right side of \eqref{unifbp}. 
To this end, let $w_n$ be the choice of the initial condition $W_0 = w \in [\frac{\sqrt{n}}{2}-n^{3/8},\frac{\sqrt{n}}{2}+n^{3/8}]$
which minimizes the probability inside the first set of braces on the right in \eqref{unifbp} (with ties broken in some predetermined deterministic way).
Now, since $\sum_{i=0}^\infty W_i \leq \frac{n}{4}$ implies that $\s_0^W < \infty$ it follows that
\begin{align}
P_W^{w_n}\left( \sum_{i=0}^\infty W_i \leq \frac{n}{4}, \, \s_M^W > \e \sqrt{n} \right) 
&\geq P_W^{w_n}\left( \sum_{i=0}^\infty W_i \leq \frac{n}{4} \right) - P_W^{w_n} \left( \s_M^W \leq \e \sqrt{n}, \, \s_0^W < \infty \right) \nonumber \\
&\geq P_W^{w_n}\left( \sum_{i=0}^\infty W_i \leq \frac{n}{4} \, \biggl| \, \s_0^W < \infty \right) P_W^{w_n}( \s_0^W < \infty ) \nonumber\\
&\qquad - P_W^{w_n}\left( \s_M^W \leq \e \sqrt{n} \, \bigl| \, \s_0^W < \infty \right) P_W^{w_n}( \s_0^W < \infty). \label{Wminlba}
\end{align}
It was shown in \cite[Proposition 16]{kzEERW} that $\lim_{n\ra\infty} n^{\d-1} P_W^{n} \left( \s_0^W < \infty \right) = \Ch$ for some constant $ \Ch>0$. 
Therefore, since $w_n \sim \sqrt{n}/2$ we can conclude from \eqref{Wminlba} and Theorem \ref{th:fbpdpl} that 
\begin{align}
& \liminf_{n\ra\infty} n^{\frac{\d-1}{2}} \left\{ \min_{|w-\frac{\sqrt{n}}{2}|\leq n^{3/8} } P_W^w\left( \sum_{i=0}^\infty W_i \leq \frac{n}{4}, \, \s_M^W > \e\sqrt{n} \right) \right\} \nonumber \\
&\qquad \geq \Ch \left\{  P_{Y_{2-\d}}^{1/2} \left( \int_0^{\s_0^{Y_{2-\d}}} \!\!\!\!\!\! Y_{2-\d}(t)\, dt \leq \frac{1}{4} \right) -  P_{Y_{2-\d}}^{1/2}( \s_0^{Y_{2-\d}} \leq \e ) \right\}. \label{Wminlb} 
\end{align}
In a similar manner one can show that 
\begin{align}
& \liminf_{n\ra\infty} \left\{ \min_{|z-\frac{\sqrt{n}}{2}|\leq n^{3/8} } P_Z^z\left( \sum_{i=0}^\infty Z_i \leq  \frac{n}{4}, \, \s_M^Z > \sqrt{n} \right) \right\} \nonumber \\
&\qquad \geq  P_{Y_{\text{-}\d}}^{1/2} \left( \int_0^{\s_0^{Y_{\text{-}\d}}} \!\!\!\!\!\! Y_{\text{-}\d}(t)\, dt \leq \frac{1}{4} \right) -  P_{Y_{\text{-}\d}}^{1/2}( \s_0^{Y_{\text{-}\d}} \leq \e ). \label{Zminlb}
\end{align}
Finally, since \eqref{Wminlb} and \eqref{Zminlb} are positive for $\e>0$ small enough, then applying \eqref{Rndev}, \eqref{Wminlb}, and \eqref{Zminlb} to \eqref{unifbp} shows that \eqref{Xreturns} holds for all $\e> 0$ small enough.

\bibliographystyle{alpha}
\bibliography{CookieRW}

\end{document}